\documentclass[reqno,a4paper,12pt]{amsart}

\usepackage[pdfpagelabels]{hyperref}
\usepackage{amssymb}
\usepackage{graphicx}

\newtheorem{theorem}{Theorem}
\newtheorem{lemma}[theorem]{Lemma}
\newtheorem{corollary}[theorem]{Corollary}

\theoremstyle{definition}

\newtheorem{example}{Example}

\theoremstyle{remark}

\newcommand{\D}{\mathbb{D}}

\newcommand{\N}{\mathbb{N}}

\newcommand{\C}{\mathbb{C}}

\renewcommand{\phi}{\varphi}

\newcommand{\MD}{\mathcal{D}}
\newcommand{\A}{\mathcal{A}}

\begin{document}

\title[Uniform separation]{Uniform separation through\\ intermediate points}
\thanks{The first author is supported by the Academy of Finland \#258125, and the second author is supported in part by the grants
MTM2011-24606 and 2014SGR 75.}

\author{Janne Gr\"ohn}
\address{Department of Physics and Mathematics, University of Eastern Finland\\ \indent P.O. Box 111, FI-80101 Joensuu, Finland}
\email{janne.grohn@uef.fi}

\author{Artur Nicolau}
\address{Departament de Matem\`atiques, Universitat Aut\`onoma de Barcelona, 08193\\ \indent Bellaterra, Barcelona, Spain}
\email{artur@mat.uab.cat }

\date{\today}

\subjclass[2010]{Primary 34C10}
\keywords{Carleson measure; linear differential equation; oscillation theory; uniform separation.}

\begin{abstract}
It is shown that a separated sequence of points in the unit disc of the complex plane
is in fact uniformly separated, if there exists a~certain intermediate sequence
whose separated subsequences are uniformly separated. This property is applied
to improve a recent result in the theory of differential equations.
\end{abstract}

\maketitle


\section{Introduction and main result}

Let $\D$ denote the open unit disc of the complex plane $\C$.
If $z,w\in\D$, then
$\varrho(z,w)=|z-w|/|1-\overline{z}{w}|$ is the pseudo-hyperbolic distance between
these points. Recall that the sequence $\{z_n\}_{n=1}^\infty\subset \D$
is called uniformly separated if
\begin{equation*}
  \inf_{k\in\N} \, \prod_{n\in\N\setminus\{k\}} \varrho(z_n,z_k)>0,
\end{equation*}
while $\{z_n\}_{n=1}^\infty\subset \D$ is said to be separated (in the pseudo-hyperbolic metric) if
there exists a~constant $\delta=\delta(\{z_n\}_{n=1}^\infty)$ with $0<\delta<1$ such that
\begin{equation} \label{eq:separated}
  \varrho(z_n,z_k) >\delta,  \quad n,k\in\N, \quad n\neq k.
\end{equation}
In this paper, separation always refers to the separation
with respect to the pseudo-hyperbolic metric.

If $z,w\in\D$ are two distinct points, then we define $\langle z,w \rangle \subset \D$ to be 
the hyperbolic segment joining $z$ and~$w$. That is,  $\langle z,w \rangle$ is a closed subarc of 
the unique hyperbolic geodesic which goes through $z\in\D$ and $w\in\D$.

The following result shows that a separated sequence of points
is in fact uniformly separated if there exists a sufficiently dispersed intermediate sequence.
In Section~\ref{sec:application} we consider an application of 
Theorem~\ref{thm:geodesics} in which the existence of the~intermediate sequence is natural.


\begin{theorem} \label{thm:geodesics}
Let $\{z_n\}_{n=1}^\infty$ be a separated sequence of points in $\D$. Suppose that there exists 
a sequence $\Lambda \subset \D$ satisfying the following properties:
\begin{enumerate}
\item[\rm (i)]
for each pair of distinct points $z_j,z_k\in \{z_n\}_{n=1}^\infty$ there corresponds a point $\xi_{j,k} \in \Lambda$ 
such that $\xi_{j,k} \in \langle z_j,z_k \rangle$;
\item[\rm (ii)]
each separated subsequence of $\Lambda$ is uniformly separated.
\end{enumerate}
Then, $\{z_n\}_{n=1}^\infty$ is uniformly separated.
\end{theorem}


\section{Auxiliary results}

The set 
\begin{equation*} 
Q = Q(I) = \big\{ re^{i\theta}  : e^{i\theta}\in I, \, 1-|I|\leq r < 1\big\}
\end{equation*}
is called a Carleson square based on the arc $I\subset \partial\D$, where $|I| = \ell(Q)$
denotes the normalized arc length of $I$ (i.e.,~$|I|$ is the Euclidean arc length of $I$ divided by $2\pi$). 
The top part of a Carleson square $Q(I)$ is 
\begin{equation*}
  T\big( Q(I) \big) = \big\{ re^{i\theta}  : e^{i\theta}\in I, \, 1-|I|\leq r <  1- |I|/2 \big\}.
\end{equation*}
For $0<K<\infty$, $K Q$ denotes the Carleson square whose base is concentric to that of $Q$, and 
for which $\ell(KQ) = K\,  \ell(Q)$.

A finite positive measure $\mu$ in $\D$ is called a Carleson measure, if there exists
a constant $0<M<\infty$ such that
\begin{equation*}
\int_{\D} |f(z)| \, d\mu(z) \leq M  \left( \sup_{0<r<1} \, \frac{1}{2\pi}  \int_0^{2\pi} |f(re^{i\theta})| \, d\theta \right)
\end{equation*}
for any analytic function $f$ in the unit disc.
Carleson proved that this holds if and only if there exists a constant $0<C<\infty$ such that
$\mu(Q) \leq C \, \ell(Q)$ for any Carleson square $Q$. It is also well-known that
a sequence $\{ z_n \}_{n=1}^\infty$ of points in $\D$ is uniformly separated if and only if
it is separated and there exists a constant $0<C<\infty$ such that 
\begin{equation} \label{eq:precarleson}
\sum_{z_n \in Q} (1-|z_n|) \leq C\, \ell(Q)
\end{equation}
for any Carleson square $Q$. For more details, we refer to \cite{G:2007}.


\begin{lemma} \label{lemma:geodesics}
Let $Q(I)$ be a Carleson square for which $0<|I|<1/2$. If $\langle z_1,z_2 \rangle$ is the hyperbolic segment
joining any points $z_1,z_2\in Q(I)$, then 
\begin{equation} \label{eq:lemmaclaim}
\langle z_1,z_2 \rangle \subset  \big\{ re^{i\theta}  : e^{i\theta}\in I, \, 1-\sqrt{1+\pi^2} \cdot |I| \leq r < 1\big\}.
\end{equation}
\end{lemma}


\begin{proof}
Without loss of generality, we may assume that $Q(I)$ is based on an~arc 
which is centered at $z=1$. Denote $\ell=|I|$. The hyperbolic segment 
$\langle z_1,z_2 \rangle$ is either a straight line segment or an~arc of a circle
$(x-x_0)^2+(y-y_0)^2=r_0^2$, which is orthogonal  to the unit circle.
By the orthogonality, $x_0^2+y_0^2=1+r_0^2$.

Let $\Pi:\D\setminus \{0\} \to \partial\D$ be the radial projection $z \mapsto z/|z|$.
By hyperbolic geometry it is clear that $\Pi(\langle z_1,z_2 \rangle) \subset I$.
The Euclidean distance between $\langle z_1,z_2 \rangle$ and the origin is as small 
as possible, if $z_1 = ( 1- \ell ) \exp( i \pi \ell )$ and $z_2 = ( 1- \ell ) \exp( - i \pi \ell )$
are the interior corners of $Q(I)$. Then $y_0=0$, and 
\begin{equation*}
\big( ( 1- \ell )  \cos( \pi \ell ) - x_0 \big)^2 + 
\big( ( 1- \ell  )  \sin( \pi \ell )  \big)^2
= r_0^2 = x_0^2 -1,
\end{equation*}
which implies
\begin{equation*}
x_0 = x_0(\ell) = \frac{1+( 1- \ell )^2}{2( 1- \ell) \cos( \pi \ell)}.
\end{equation*}
Now, $\langle z_1,z_2 \rangle$ intersects the real axis at $X_0  = X_0(\ell) = x_0(\ell) - \sqrt{x_0(\ell)^2-1}$,
which is precisely the point on $\langle z_1,z_2 \rangle$ which is closest to the origin.
Since
\begin{equation*}
\sup_{0<\ell <1/2} \, \frac{1-X_0(\ell)}{\ell}  = \sqrt{1 + \pi^2} \approx 3.30,
\end{equation*}%
the inclusion \eqref{eq:lemmaclaim} follows. 
\end{proof}

The next lemma introduces a partition of arcs, which plays a significant role
in our construction.


\begin{lemma} \label{lemma:partition}
Let $I\subset \partial\D$ be a closed arc for which $0<|I|<1/8$. Suppose that $0<\varepsilon<1$, 
and let $\{\xi_k\}_{k=1}^K\subset \D$  be a finite collection of points 
such that $\varrho(\xi_j,\xi_k)>\varepsilon$ for any $j\neq k$.
Suppose that $0\leq r < 1$ is sufficiently large to satisfy $\max\, \{ |\xi_k| \}_{k=1}^K \leq r$ and $1-r \leq |I|$.
Then, there exist a constant $\eta=\eta(\varepsilon)$ with $0<\eta<1$ and a~partition $I=\bigcup_{n=1}^N I_n$, 
which divides $I$ into $N\leq 8K+8$ closed subarcs (having pairwise disjoint interiors) such  that 
\begin{enumerate}
\item[\rm (i)]
$|I_n| \geq (1-r)/64$;

\item[\rm (ii)]
any hyperbolic segment $\gamma$, which joins two points in $Q(I_n)$, 
satisfies $\varrho(\gamma, \{\xi_k\}_{k=1}^K) > \eta$;
\end{enumerate}
for all $n=1,\dotsc,N$.
\end{lemma}


\begin{proof}

Since $\{\xi_k\}_{k=1}^K\subset \D$ is separated,
there exists a constant $\mu=\mu(\varepsilon)$ with $0<\mu<1/2$ such that the Euclidean discs
$\overline{D}_k = \{ z \in \C : |z - \xi_k| \leq \mu(1-|\xi_k|) \}$ are pairwise disjoint for $k=1,\dotsc,K$.
We write $\MD = \bigcup_{k=1}^K \overline{D}_k$ for short,
and state a property which follows from Lemma~\ref{lemma:geodesics}. This property will be
referred as the \emph{auxiliary claim}: There exists a constant $\eta=\eta(\varepsilon)$ with $0<\eta<1$ such that,
if $I^\star$ is a subarc of $I$ for which the interior of
$\{r e^{i\theta} : e^{i\theta}\in I^\star, \, 1 -  4\, |I^\star| \leq r < 1\}$ does not
intersect $\MD$, then any hyperbolic segment $\gamma$ joining two points in the Carleson square $Q(I^\star)$ satisfies
$\varrho(\gamma,\{\xi_k\}_{k=1}^K)\geq \eta$.

Denote $I=[a,b]$, where $0<b-a<2\pi$. That is, the interval $[a,b]$ is identified with 
the arc~$I= \{ e^{i\theta} \in \partial\D : a\leq \theta \leq b\}$.
If $I^\star$ is a subarc of $I$, then the collection of subarcs $\{ I_1, \dotsc,I_M\}$ of $I^\star$ with pairwise disjoint
interiors is called an admissible partition of $I^\star$ provided that $I^\star=\bigcup_{n=1}^{M} I_n$,
and the conditions~(i) and (ii) are fulfilled. 

If $Q( [a,b]) \cap \MD = \emptyset$, then we split $I$ into four subarcs $I_1,\dotsc,I_4$ of equal length.
The collection $\{ I_1, \dotsc,I_4\}$ is an admissible partition of~$I$ by the auxiliary claim.
Otherwise $Q( [a,b]) \cap \MD \neq \emptyset$, and we let $a_1$ to be the smallest value such that $a<a_1\leq b$ and 
\begin{equation} \label{eq:notempty}
  Q([a,a_1]) \cap \MD \neq \emptyset.
\end{equation}
We split $[a,a_1]$ into four subarcs $I_1,\dotsc,I_4$
of equal length. By \eqref{eq:notempty}, there is a~point $z_1\in Q([a,a_1]) \cap \MD$. Then
$a_1-a \geq 2\pi (1-|z_1|) \geq 2\pi(1-\mu) \, (1-r)$, which implies
\begin{equation*}
  |I_n| = \frac{1}{4} \cdot \frac{a_1-a}{2\pi} \geq \frac{1-\mu}{4} \, (1-r), \quad n=1,\dotsc,4,
\end{equation*} 
and hence the arcs $I_1, \dotsc, I_4$ satisfy (i). The arcs $I_1, \dotsc, I_4$ satisfy (ii) by the auxiliary claim.

We continue inductively. Suppose that we have a sequence $\{ a_j \}_{j=1}^M\subset [a,b]$, $M\in\N$,
which determines an admissible partition $\{I_1, \dotsc, I_{4M}\}$ of $[a,a_M]$.
There are four possible cases:
\begin{enumerate}

\item[\rm (I)]
If $a_M=b$, then we stop the inductive process.

\item[\rm (II)]
If $a_M<b$ and $Q( [a_M,b]) \cap \MD \neq \emptyset$, then let  $a_{M+1}$ be the smallest value such that 
$a_M<a_{M+1}\leq b$ and  
\begin{equation} \label{eq:notempty2}
  Q( [a_M,a_{M+1}] ) \cap \MD \neq \emptyset. 
\end{equation}
We split $[a_M,a_{M+1}]$ into four subarcs $I_{4M+1},\dotsc, I_{4(M+1)}$ of equal length.
According to \eqref{eq:notempty2}, there exists a point $z_{M+1}\in Q([a_M,a_{M+1}]) \cap \MD$. Then
$a_{M+1}-a_M \geq 2\pi (1-|z_{M+1}|) \geq 2\pi (1-\mu) \, (1-r)$, which implies
\begin{equation*}
  |I_{4M+n}| = \frac{1}{4} \cdot \frac{a_{M+1}-a_M}{2\pi} \geq \frac{1-\mu}{4} \, (1-r), \quad n=1,\dotsc,4,
\end{equation*} 
and hence the arcs $I_{4M+1}, \dotsc, I_{4(M+1)}$ satisfy (i). 
Moreover, the property (ii) holds by the auxiliary claim.
In conclusion, we have a sequence $\{ a_j \}_{j=1}^{M+1}$, which determines 
an admissible partition $\{I_1, \dotsc, I_{4(M+1)}\}$ of $[a,a_{M+1}]$. We proceed with the induction.

\item[\rm (III)]
If  $a_M \leq b-2 \pi (1-\mu) (1-r)/8$ and $Q( [a_M,b]) \cap \MD = \emptyset$, then define $a_{M+1}=b$
and split $[a_M,b]$ into four subarcs $I_{4M+1},\dotsc, I_{4(M+1)}$ of equal length. We have
\begin{equation*}
|I_{4M+n}| = \frac{1}{4} \cdot \frac{b-a_M}{2\pi} \geq \frac{1-\mu}{32} \, (1-r), \quad n=1,\dotsc,4,
\end{equation*} 
and hence the arcs $I_{4M+1}, \dotsc, I_{4(M+1)}$ satisfy (i).
The property (ii) holds by the auxiliary claim.
In conclusion, we have a sequence $\{ a_j \}_{j=1}^{M+1}$, which determines 
an admissible partition $\{ I_1, \dotsc, I_{4(M+1)}\}$ of $[a,b]$.  We stop the inductive process.

\item[\rm (IV)]
If  $ b-2 \pi (1-\mu) (1-r)/8 < a_M < b$ and $Q( [a_M,b]) \cap \MD = \emptyset$, 
then define 
\begin{equation*}
\qquad \quad a_M^\star = b-2\pi \, \frac{1-\mu}{8} \,  (1-r), \quad I_{4M}^\star = \left[ a_{M-1}+3 \, \frac{a_{M}-a_{M-1}}{4},\, a_M^\star \right].
\end{equation*}
Since $|I_{4M}| \geq (1-\mu) (1-r)/4$, we have
\begin{align*}
\qquad \quad \big| I_{4M}^\star \big| 
  = \big| I_{4M} \big| - \frac{a_M - a_M^\star}{2\pi}
  \geq \big| I_{4M} \big| - \frac{1-\mu}{8} \, (1-r)
 \geq  \frac{1-\mu}{8} \, (1-r).
\end{align*}
We conclude that $I_{4M}^\star$ satisfies (i).
It is clear that~$I_{4M}^\star$ satisfies the property  (ii), since it is a subarc of $I_{4M}$. 

Define $a_{M+1}= b$ and $I_{4M+1} = [a_M^\star,b]$. The arc $I_{4M+1}$ satisfies the estimate~(i),
since $|I_{4M+1}| = (1-\mu) (1-r)/8$. Now
\begin{equation*}
\qquad \quad 1 - 4\,  |I_{4M+1}| 
 \geq   r + \mu  (1-r) \geq |\xi_k| + \mu\, (1-|\xi_k|), \quad k=1,\dotsc,K.
\end{equation*}
The property (ii) for $I_{4M+1}$ follows from the auxiliary claim. In conclusion, we have 
a sequence $\{ a_1, \dotsc, a_{M-1}, a_M^\star, a_{M+1} \}$, which determines 
an admissible partition $\{I_1, \dotsc, I_{4M-1}, I_{4M}^\star, I_{4M+1} \}$ of $[a,b]$.  We stop the inductive process.
\end{enumerate}

\medskip

The inductive process above produces a finite collection of points
$$a=a_0 < a_1 < \dotsb <a_{J-1} < a_J = b,$$
which determines an admissible partition $\{I_1, \dotsc, I_N\}$ of $I=[a,b]$.
Each arc $\A_j = [a_{j-1},a_j]$ for $j=1,\dotsc,J$ 
is partitioned into at most four subarcs $I_n$, and hence $N\leq 4J$. It suffices to bound
$J$ in terms of $K$ to complete the proof.

Without loss of generality, we may assume that $1 - |\xi_k| \leq 4 \,|I|$ for all $k=1,\dotsc,K$.
This is because any $\xi_k$, for which $1 - |\xi_k| > 4 \,|I|$, is separated from any 
hyperbolic segment joining two points in $Q(I)$ by Lemma~\ref{lemma:geodesics}.
First, note that each Euclidean disc $\overline{D}_k$ for fixed $1\leq k \leq K$ can meet at most 
two Carleson squares $Q(\A_j)$.
Second, we know that every $Q(\A_j)$ for $0\leq j \leq J-2$ meets some disc $\overline{D}_k$ by construction.
This gives $J-2 \leq 2K$, which implies $N \leq 4J \leq  8K+8$.
\end{proof}

The last auxiliary result shows that, if there are three hyperbolic segments
of certain type, then there is no point in their union which is
simultaneously (pseudo-hyperbolically) close to all of them.


\begin{lemma} \label{lemma:number}
  Let $Q$ be a Carleson square for which $0<\ell(Q)<1/4$. Let $\gamma_1$, $\gamma_2$ and $\gamma_3$ 
  be three hyperbolic segments connecting points in $Q$ such that the radial projections 
  $\Pi(\gamma_1)$, $\Pi(\gamma_2)$ and $\Pi(\gamma_3)$ have pairwise disjoint interiors, and
  the hyperbolic segments satisfy the geometric property
  \begin{equation*}
    \frac{1}{K} \, \max_{\xi \in\gamma_j} \big\{ 1- |\xi| \big\} 
    \leq \big| \Pi(\gamma_j) \big| 
    \leq K \, \max_{\xi \in\gamma_j} \big\{ 1- |\xi| \big\}, \quad j=1,2,3,
  \end{equation*}
  for some constant $1<K<\infty$. Then, there exists a constant $\mu=\mu(K)$ with $0<\mu<1$
  such that 
  \begin{equation*}
    \varrho(\xi,\gamma_1) + \varrho(\xi,\gamma_2) + \varrho(\xi,\gamma_3) \geq \mu, \quad \xi\in \gamma_1 \cup \gamma_2 \cup \gamma_3.
  \end{equation*}
\end{lemma}

\begin{proof}
Without loss of generality, we may assume that the normalized lengths of the radial
projections satisfy $| \Pi(\gamma_1) | \leq \min \{ |\Pi(\gamma_2)|, |\Pi(\gamma_3)| \}$.
Since $\ell(Q)<1/4$, all hyperbolic segments connecting any two
points in $Q$ are uniformly bounded away from the origin by Lemma~\ref{lemma:geodesics}.

If $\Pi(\gamma_2)$ and $\Pi(\gamma_3)$ are on the same side of $\Pi(\gamma_1)$, then we may assume
that $\Pi(\gamma_2)$ is located in the middle of  $\Pi(\gamma_1)$ and $\Pi(\gamma_3)$. Let
$A = \{ r e^{i\theta(A)} : 0<r<1\}$ be a radial segment, where $e^{i\theta(A)}\in\partial\D$ 
lies strictly between the interiors of $\Pi(\gamma_2)$ and $\Pi(\gamma_3)$.
In this case $\gamma_1$ and $A$,
and hence also $\gamma_1$ and $\gamma_3$, are separated by a constant depending on $K$.

If $\Pi(\gamma_2)$ and $\Pi(\gamma_3)$ are on the opposite sides of $\Pi(\gamma_1)$, then
let $\arg{z}$ denote a continuous branch of the argument defined in $Q$.
We may assume that the hyperbolic segments satisfy the ordering $\gamma_2,\gamma_1,\gamma_3$
with respect to the increasing argument.
Let $B = \{ r e^{i\theta(B)} : 0<r<1\}$ be a~radial segment, where 
$e^{i\theta(B)}\in\partial\D$ lies strictly between the interiors of $\Pi(\gamma_1)$ and $\Pi(\gamma_3)$.
Moreover, let $C = \{ r e^{i\theta(C)} : 0<r<1\}$ be a radial segment, where $e^{i\theta(C)}\in\partial\D$ 
is the midpoint of $\Pi(\gamma_1)$. In this case
any point $\xi\in \gamma_1$ for which $\arg{\xi} \leq \theta(C)$ is 
separated from~$B$, and hence is also separated from~$\gamma_3$, by a constant depending on $K$. Analogously
any point $\xi\in \gamma_1$ for which $\arg{\xi} > \theta(C)$ is separated from $\gamma_2$ by a constant depending on $K$.
\end{proof}


\section{Proof of Theorem~\ref{thm:geodesics}}

Note that
$\{z_n\}_{n=1}^\infty$ can be divided into finitely many subsequences 
$\mathcal{Z}_1, \dotsc, \mathcal{Z}_M$ such that
for any $j=1,\dotsc, M$ the following two conditions hold: top part of any Carleson square 
contains at most one point of~$\mathcal{Z}_j$, and there exists an~integer $m=m(j) \in \{ 0,1, \dotsc, 6\}$ such that
\begin{equation*}
\mathcal{Z}_j \subset \bigcup_{k=0}^\infty \big\{ z\in \D : 2^{-m-7k} < 1-|z| \leq 2^{-m-(7k-1)} \big\}.
\end{equation*}
It is sufficient to prove that each subsequence $\mathcal{Z}_j$ is uniformly separated. Hence, 
without loss of generality we may assume that the following two conditions hold:
\begin{enumerate}
\item[\rm (A)]
$\{z_n\}_{n=1}^\infty$ satisfies \eqref{eq:separated} for $0<\delta<1$, where $\delta$ is
so large that the top part of each Carleson square contains
at most one point from $\{ z_n \}_{n=1}^\infty$;

\item[\rm (B)]
$\{z_n\}_{n=1}^\infty$ satisfies
\begin{equation*} 
\{ z_n \}_{n=1}^\infty \subset \bigcup_{k=1}^\infty \big\{ z\in\D : 2^{-7k} < 1-|z| \leq 2^{-(7k-1)}  \big\}.
\end{equation*}
\end{enumerate}

We proceed to show that there exists a constant $0<C<\infty$ such that
\eqref{eq:precarleson} holds for any Carleson square $Q$ for which $0<\ell(Q)<1/8$.
Let $Q$ be a such Carleson square, and let $\arg z$ be a continuous branch of the argument defined in $Q$. 
By means of an inductive argument, we divide $\{z_n\}_{n=1}^\infty$
into subsequences such that
\begin{equation*}
\{ z_n \}_{n=1}^\infty \cap Q = \bigcup_{j=1}^\infty \Big( M^{(j)} \cup S^{(j)} \Big),
\end{equation*}
where the subsequences $M^{(j)}$ and $S^{(j)}$ satisfy the following properties:
\begin{enumerate}


\item[\rm (a)]
Concerning $S^{(j)}$, we have 
\begin{align*} 
  \sum_{z_n\in S^{(1)}} (1-|z_n|) & \leq 4 \, \ell(Q),\\
  \sum_{z_n\in S^{(j)}} (1-|z_n|) & \leq \frac{1}{2}  \sum_{z_n\in M^{(j-1)}} (1- |z_n|), \qquad j>1.
\end{align*}


\item[\rm (b)]
Concerning $M^{(j)}$, we construct sequences $\Lambda_Q^{(j)}$ such that
\begin{equation*}
\Lambda_Q = \bigcup_{j=1}^\infty  \Lambda_Q^{(j)} \subset \big( \Lambda \cap 4Q \big)
\end{equation*}
can be represented as a union of two separated subsequences, and
\begin{equation} \label{eq:Mestimate}
\sum_{z_n\in M^{(j)}} (1-|z_n|) 
\leq 6 \, \sum_{\xi \in \Lambda_Q^{(j)}} \big( 1- |\xi| \big), \quad j\in\N.
\end{equation}
\end{enumerate}
It is possible that some of these subsequences are empty,
and in those cases the corresponding sums in (a) and (b) are zero by definition.
It is clear that the properties (a) and (b) imply
\begin{align*}
  \sum_{z_n\in Q} (1-|z_n|) 
   & \leq \sum_{z_n\in S^{(1)}} (1-|z_n|) + \frac{3}{2}\,  \sum_{j=1}^\infty \sum_{z_n\in M^{(j)}} (1-|z_n|) \\
  & \leq 4\, \ell(Q) + 9\,  \sum_{\xi \in \Lambda_Q} (1-|\xi|),
\end{align*}
which finishes the proof, since $\Lambda_Q$ can be represented as a~union of 
two uniformly separated sequences by (b) and the assumption (ii).

Let us now proceed with the construction of the subsequences $S^{(j)}$ and $M^{(j)}$.
Consider the dyadic annuli
\begin{equation*}
C_k = \big\{ z \in \D :  2^{-k} < 1- |z| \leq 2^{-(k-1)} \big\}, \quad k\in\N.
\end{equation*}


\subsection*{First step}

Let $Z^{(1)} =  \{z_n\}_{n=1}^\infty \cap Q$, and consider the arc $I^{(1)} = \overline{Q} \cap \partial\D$,
where~$\overline{Q}$ is the closure of the Carleson square $Q$.
Let $S^{(1)}$ be the (possibly empty) subsequence of those points $z_n \in Z^{(1)}$ for which
there exists a constant $k(z_n)\in\N$ such that $z_n$ is the only point of $Z^{(1)}$ belonging to $C_{k(z_n)}$.
Points in $S^{(1)}$ are called single points of the first generation. It is easy to see that
\begin{equation*} 
\sum_{z_n \in S^{(1)}} (1-|z_n|) = \sum_{k=1}^\infty \sum_{z_n \in S^{(1)} \cap C_k} (1-|z_n|) \leq 4 \, \ell(Q),
\end{equation*}
which proves (a) for $j=1$.

By construction any point in $Z^{(1)} \setminus S^{(1)}$ has
at least one partner in the same dyadic annulus. More precisely, 
\begin{equation} \label{eq:firstint}
\big( Z^{(1)} \setminus S^{(1)} \big) \cap C_k
\end{equation}
is either empty or contains at least two points for any $k\in\N$. If \eqref{eq:firstint} is empty for all $k\in\N$,
then define $M^{(1)} = \emptyset$, and move on to the second step of the inductive process.
Otherwise, define
\begin{equation*}
k(I^{(1)}) = \min \Big\{ k\in\N : \big( Z^{(1)} \setminus S^{(1)} \big) \cap C_k \neq \emptyset \Big\},
\end{equation*}
and write
\begin{equation*}
M^{(1)} = \big\{ z_n^{(1)} : n=1,\dotsc, N^{(1)} \big\} =  \big( Z^{(1)} \setminus S^{(1)} \big) \cap C_{k(I^{(1)})}.
\end{equation*}
Order the points in $M^{(1)}$ by increasing argument. Now
$\arg(z_{n}^{(1)}) < \arg(z_{n+1}^{(1)})$ for $n=1,\dotsc,N^{(1)} - 1$, where the inequality is strict 
by the reduction~(A). For the same values of $n$, let $\gamma_n^{(1)} = \langle z_{n}^{(1)}, z_{n+1}^{(1)}\rangle$ 
be the hyperbolic segments joining $z_{n}^{(1)}$ and $z_{n+1}^{(1)}$, and consider the points $\xi_{n}^{(1)}\in\gamma_{n}^{(1)} \cap \Lambda$
given by the assumption. 

It is clear that the subsequence
$\{ \xi_{2n-1}^{(1)} : n=1,\dotsc, \lfloor N^{(1)}/2 \rfloor \}$
is separated, where 
$\lfloor x \rfloor$ denotes the integer part of $x$.
We also point out that there exists a~constant $1<K<\infty$
such that the normalized lengths of the radial
projections of  the hyperbolic segments 
$\gamma_{2n-1}^{(1)} = \langle z_{2n-1}^{(1)}, z_{2n}^{(1)}\rangle$ satisfy
\begin{equation} \label{eq:nicearcs}
  \frac{1}{K} \, \max_{\xi\in \gamma_{2n-1}^{(1)}} \big\{ 1-|\xi| \big\} 
  \leq \big| \Pi\big(\gamma_{2n-1}^{(1)}\big) \big| 
  \leq K \max_{\xi\in \gamma_{2n-1}^{(1)}} \big\{ 1-|\xi| \big\}
\end{equation}
for all $n=1, \dotsc, \lfloor N^{(1)}/2 \rfloor$.

Since $\gamma_{2n-1}^{(1)}$ is a hyperbolic segment joining points in $C_{k(I^{(1)})}$, we have
\begin{equation} \label{eq:dist1}
1-\big|\xi_{2n-1}^{(1)}\big|> 2^{-k(I^{(1)})}, \quad n=1, \dotsc, \lfloor N^{(1)}/2 \rfloor.
\end{equation}
Hence,
\begin{equation} \label{eq:lower1}
\sum_{n=1}^{\lfloor N^{(1)}/2 \rfloor} \big( 1-\big| \xi_{2n-1}^{(1)} \big| \big)
  \geq 2^{-k(I^{(1)})} \bigg\lfloor \frac{N^{(1)}}{2} \bigg\rfloor 
  \geq \frac{1}{6} \, \sum_{z_n\in M^{(1)}} (1-|z_n|).
\end{equation}
Define $\Lambda_Q^{(1)} = \{ \xi_{2n-1}^{(1)} : n=1,\dotsc, \lfloor N^{(1)}/2 \rfloor\}$, 
which is known to be a  separated sequence. Now \eqref{eq:lower1} proves \eqref{eq:Mestimate} for $j=1$. 
In conclusion, we have proved the properties (a) and (b) for $j=1$.


\subsection*{Second step}
If 
\begin{equation} \label{eq:bsplit}
Z^{(1)} \setminus \big( S^{(1)} \cup M^{(1)} \big)
\end{equation}
is empty, then define $S^{(j)} = M^{(j)} = \emptyset$ for all $j\in\N\setminus \{1\}$, and stop the inductive process. 
Otherwise, we proceed to split \eqref{eq:bsplit} into subsequences, where the
number of subsequences is at most a constant multiple of $N^{(1)}$ (i.e., the number of elements in $M^{(1)}\neq \emptyset$).
We apply Lemma~\ref{lemma:partition} to $I^{(1)}$ and $\Lambda_Q^{(1)}$ for $r=1 - 2^{-k(I^{(1)})}$, see \eqref{eq:dist1}. 
Lemma~\ref{lemma:partition} produces a partition of $I^{(1)}$ into arcs $I^{(2)}_p$,
\begin{equation} \label{eq:newpart}
I^{(1)} = \bigcup_{p=1}^{P_2}  I_p^{(2)}, \quad P_2 \leq 8 \lfloor N^{(1)}/2 \rfloor + 8 \leq 8\,  N^{(1)},
\end{equation}
with pairwise disjoint interiors, where
\begin{equation*}
\big| I_p^{(2)} \big| \geq (1-r)/64 =  2^{-k(I^{(1)})}/64 , \quad p=1,\dotsc, P_2.
\end{equation*}
By the reduction (B), we obtain
\begin{equation*} 
Z^{(1)} \setminus \big( S^{(1)} \cup M^{(1)} \big) \subset \bigcup_{p=1}^{P_2} Q\big( I_p^{(2)} \big).
\end{equation*}

Define
\begin{equation*}
Z^{(2)}_p = Q(I_p^{(2)}) \cap \Big( Z^{(1)} \setminus \big( S^{(1)} \cup M^{(1)} \big) \Big), \quad p=1,\dotsc,P_2,
\end{equation*}
and note that $Z^{(1)} \setminus ( S^{(1)} \cup M^{(1)} ) = \bigcup_{p=1}^{P_2} Z_p^{(2)}$.
We proceed to repeat the first step for each $p=1,\dotsc, P_2$ with $Z^{(1)} \subset Q$ replaced by $Z^{(2)}_p \subset Q(I^{(2)}_p)$.

Fix $p=1,\dotsc, P_2$. If $Z^{(2)}_p=\emptyset$, then define $S^{(2)}_p = M^{(2)}_p=\emptyset$ and
turn to consider another value of $p$.
Otherwise $Z^{(2)}_p\neq \emptyset$, and then define $S^{(2)}_p$ to be the (possibly empty) subsequence 
of those points $z_n \in Z^{(2)}_p$ for which
there exists a constant $k(z_n)\in\N$ such that 
$z_n$ is the only point of $Z^{(2)}_p$ belonging to $C_{k(z_n)}$. Now
\begin{equation} \label{eq:singles2}
\sum_{z_n \in S^{(2)}_p} (1-|z_n|) = \sum_{k=1}^\infty \sum_{z_n \in S^{(2)}_p \cap C_k} (1-|z_n|) \leq \frac{1}{16} \, 2^{-k(I^{(1)})},
\end{equation}
since any point in $Z^{(2)}_p$ has modulus larger than $1 - 2^{-k(I^{(1)})}/64$.

By construction any point in $Z^{(2)}_p \setminus S^{(2)}_p$ has
at least one partner in the same dyadic annulus. That is, 
\begin{equation} \label{eq:int2}
\big( Z^{(2)}_p \setminus S^{(2)}_p \big) \cap C_k
\end{equation}
is either empty or contains at least two points for any $k\in\N$.
If \eqref{eq:int2} is empty for all $k\in\N$, then define $M^{(2)}_p=\emptyset$ and move on to consider
another value of~$p$. Otherwise, define
\begin{equation*}
k(I^{(2)}_p) = \min \Big\{ k\in\N : \big( Z^{(2)}_p \setminus S^{(2)}_p \big) \cap C_k \neq \emptyset \Big\},
\end{equation*}
and write
\begin{equation*}
M^{(2)}_p = \big\{ z_n^{(2,p)} : n=1, \dotsc, N_p^{(2)} \big\} =  \big( Z^{(2)}_p \setminus S^{(2)}_p \big) \cap C_{k(I^{(2)}_p)}.
\end{equation*}
Order the points in $M^{(2)}_p$ by increasing argument. Consequently, we have
$\arg(z_{n}^{(2,p)}) < \arg(z_{n+1}^{(2,p)})$ for $n=1,\dotsc,N_p^{(2)} - 1$, where the inequality is strict 
by~(A). For the same values of $n$, let $\gamma_n^{(2,p)} = \langle z_{n}^{(2,p)}, z_{n+1}^{(2,p)}\rangle$ 
be the hyperbolic segments joining $z_{n}^{(2,p)}$ and $z_{n+1}^{(2,p)}$, and consider the points $\xi_{n}^{(2,p)}\in\gamma_{n}^{(2,p)} \cap \Lambda$
given by the assumption. 

It is clear that the subsequence 
$\{ \xi_{2n-1}^{(2,p)} : n=1,\dotsc, \lfloor N_p^{(2)}/2 \rfloor\}$
is separated. Consequently, Lemma~\ref{lemma:partition} guarantees that 
\begin{equation} \label{eq:sepap}
\big\{ \xi_{2n-1}^{(2,p)} : n=1,\dotsc, \lfloor N_p^{(2)}/2 \rfloor \big\} \cup \Lambda_Q^{(1)}
\end{equation}
is separated. Corresponding to the situation in \eqref{eq:nicearcs}, the normalized lengths of the radial
projections of the hyperbolic segments 
$\gamma_{2n-1}^{(2,p)} = \langle z_{2n-1}^{(2,p)}, z_{2n}^{(2,p)}\rangle$ satisfy
\begin{equation} \label{eq:nicearcs2}
  \frac{1}{K} \, \max_{\xi\in \gamma_{2n-1}^{(2,p)}} \big\{ 1-|\xi| \big\} 
  \leq \big| \Pi\big(\gamma_{2n-1}^{(2,p)}\big) \big| 
  \leq K \max_{\xi\in \gamma_{2n-1}^{(2,p)}} \big\{ 1-|\xi| \big\}
\end{equation}
for all $n=1, \dotsc, \lfloor N_p^{(2)}/2 \rfloor$. Here $K$ is the same constant as in \eqref{eq:nicearcs}.

Since $\gamma_{2n-1}^{(2,p)}$ are hyperbolic segments joining points in $C_{k(I^{(2)}_p)}$, we have
\begin{equation*} \label{eq:dist2}
1-\big|\xi_{2n-1}^{(2,p)}\big|> 2^{-k(I^{(2)}_p)}, \quad n=1, \dotsc, \lfloor N_p^{(2)}/2 \rfloor.
\end{equation*}
Hence, as in \eqref{eq:lower1},
\begin{equation} \label{eq:sumest22}
\sum_{n=1}^{\lfloor N_p^{(2)}/2 \rfloor} \big( 1-\big| \xi_{2n-1}^{(2,p)} \big| \big)
   \geq  2^{-k(I^{(2)}_p)} \, \bigg\lfloor \frac{N_p^{(2)}}{2} \bigg\rfloor
  \geq \frac{1}{6} \, \sum_{n=1}^{N_p^{(2)}} (1-|z_n^{(2,p)}|).
\end{equation}

Define
\begin{equation*}
S^{(2)} = \bigcup_{p=1}^{P_2} S^{(2)}_p, \quad 
M^{(2)} = \bigcup_{p=1}^{P_2} M^{(2)}_p, \quad 
\Lambda_Q^{(2)} = \bigcup_{p=1}^{P_2}  \big\{ \xi_{2n-1}^{(2,p)} : n=1, \dotsc, \lfloor N_p^{(2)}/2 \rfloor \big\},
\end{equation*}
where the points in $S^{(2)}$ are said to be singles of  the second generation.
Then, by means of \eqref{eq:newpart} and \eqref{eq:singles2},
\begin{equation*}
\sum_{z_n \in S^{(2)}} (1-|z_n|)
        = \sum_{p=1}^{P_2} \sum_{z_n \in S^{(2)}_p} (1-|z_n|)
         \leq \frac{1}{2}\,  N^{(1)}  2^{-k(I^{(1)})} 
        \leq \frac{1}{2}  \, \sum_{n=1}^{N^{(1)}} (1-|z_n^{(1)}|),
 \end{equation*}
which proves (a) for $j=2$. By \eqref{eq:sumest22}, we deduce
\begin{equation*}
  \sum_{z_n\in M^{(2)}} (1-|z_n|) 
    = \sum_{p=1}^{P_2} \sum_{n=1}^{N_p^{(2)}} (1-|z_n^{(2,p)}|)
   \leq 6 \, \sum_{p=1}^{P_2} \sum_{n=1}^{\lfloor N_p^{(2)}/2 \rfloor} \big( 1-\big| \xi_{2n-1}^{(2,p)} \big| \big),
 \end{equation*}
which proves \eqref{eq:Mestimate} for $j=2$. It remains to show that
$\Lambda_Q^{(1)} \cup \Lambda_Q^{(2)}$ is a~union of two
separated sequences. By \eqref{eq:sepap}, it suffices to know that
\begin{equation*} 
\big\{ \xi_{2n-1}^{(2,p_1)} : n=1,\dotsc, \lfloor N_{p_1}^{(2)}/2 \rfloor \big\} 
\, \cup \, \big\{ \xi_{2n-1}^{(2,p_2)} : n=1,\dotsc, \lfloor N_{p_2}^{(2)}/2 \rfloor \big\}
\end{equation*}
is a union of two separated sequences  for any  pair of indices $1 \leq p_1 < p_2 \leq P_2$.
This is guaranteed by~\eqref{eq:nicearcs2} and Lemma~\ref{lemma:number}. 
In conclusion, we have proved (a) and (b) for $j=2$. 

We continue inductively. In the next step we only need to apply Lemma~\ref{lemma:partition}
to those intervals $I^{(2)}_p$ for which $M^{(2)}_p\neq \emptyset$. The inductive process
gives the estimates in (a) and (b), and shows that $\Lambda_Q = \bigcup_{j=1}^\infty \Lambda_Q^{(j)}$ 
is  a~union of two separated sequences:
\begin{enumerate}
\item[\rm (i)]
If $\xi_1,\xi_2\in\Lambda_Q$ are distinct points which are situated on hyperbolic segments $\gamma_1$ and $\gamma_2$
such that $\rm{interior}(\Pi(\gamma_1)) \cap \rm{interior}(\Pi(\gamma_2)) \neq \emptyset$, then
$\xi_1$ and~$\xi_2$ are separated by Lemma~\ref{lemma:partition}.

\item[\rm (ii)]
If $\xi_1,\xi_2\in\Lambda_Q$ are distinct points which are situated on hyperbolic segments whose
 radial projections have disjoint interiors, then $\xi_1$ and $\xi_2$
may be pseudo-hyperbolically close. But if this happens, then all points in $\Lambda_Q\setminus \{ \xi_1, \xi_2 \}$ are
separated from $\{ \xi_1, \xi_2 \}$ by Lemmas~\ref{lemma:partition} and \ref{lemma:number}. 
\end{enumerate}
The assertion of Theorem~\ref{thm:geodesics} follows.


\section{An application} \label{sec:application}

Let $f$ be a non-trivial ($f\not\equiv 0$) solution of the linear differential equation
\begin{equation} \label{eq:de2}
f''+Af=0
\end{equation}
with an analytic coefficient function $A$. Let $0<p<\infty$, and suppose that 
$|A(z)|^p (1-|z|^2)^{2p-1} \, dm(z)$ is a Carleson measure.
Here $dm(z)$ is the element of the Lebesgue area measure. Note that the coefficient $A$ satisfies
\begin{equation} \label{eq:h2}
\sup_{z\in\D} \, (1-|z|^2)^2 |A(z)| < \infty
\end{equation}
by the subharmonicity of $|A|^p$.

If $\{z_n\}_{n=1}^\infty$ is the zero-sequence of $f$, then \eqref{eq:h2} and \cite[Theorem~3]{S:1955} imply that
$\{z_n\}_{n=1}^\infty$ is separated. Moreover, it is implicit in the proof of \cite[Theorem~I]{N:1949}
that for each pair of distinct zeros $z_j$ and $z_k$ there exists a point $\xi_{j,k} \in \langle z_j, z_k \rangle \subset\D$ at which
$(1-|\xi_{j,k}|^2)^2 |A(\xi_{j,k})| >1$. Define
\begin{equation*}
\Lambda = \Big\{ \xi_{j,k} \in \langle z_j,z_k \rangle : z_j,z_k\in\{z_n\}_{n=1}^\infty, z_j\neq z_k \Big\}.
\end{equation*}

The property (i) in Theorem~\ref{thm:geodesics} is given by the construction. To see that 
the property (ii) holds, let $\Lambda' = \{ \xi_n' \}_{n=1}^\infty$ be any separated subsequence of $\Lambda$ with
the separation constant $0<\delta<1$. Consequently, there exists a constant $\eta=\eta(\delta)$
with $0<\eta<1$ such that the Euclidean discs $D_n = D(\xi_n',\eta(1-|\xi_n'|))$ are pairwise disjoint, and
$D_n \subset 2Q$ whenever $\xi_n'\in Q$.

The subharmonicity of $|A|^p$ implies that there exists a~constant $C = C(\delta,p)$ 
with $0<C<\infty$ such that
\begin{align*}
\sum_{\xi_n'\in Q} (1-|\xi_n'|) & \leq \sum_{\xi_n'\in Q} (1-|\xi_n'|^2)^{2p+1} |A(\xi_n')|^p\\
& \leq \sum_{\xi_n'\in Q} C  \int_{D_n} |A(z)|^p (1-|z|^2)^{2p-1} \, dm(z)\\
& \leq C \int_{2Q} |A(z)|^p (1-|z|^2)^{2p-1} \, dm(z)
\end{align*}
for all Carleson squares $Q$. Hence $\{z_n\}_{n=1}^\infty$ is uniformly separated.

We have proved the following result, which reduces to \cite[Theorem~1]{GNR:preprint} in
the special case $p=1$.


\begin{corollary} \label{cor:de}
If $A$ is analytic in $\D$, and $|A(z)|^p (1-|z|^2)^{2p-1} \, dm(z)$ is a~Carleson measure
for some $0<p<\infty$, then the zero-sequence of each non-trivial solution $f$ of \eqref{eq:de2}
is uniformly separated.
\end{corollary}

By the well-known connection between solutions of \eqref{eq:de2}
and locally univalent meromorphic functions \cite[p.~546]{N:1949}, Corollary~\ref{cor:de} can be stated in 
the following equivalent form. Recall that, if $w$ is meromorphic and locally univalent, then its Schwarzian derivative
\begin{equation*}
S_w = \left( \frac{w''}{w'} \right)' - \frac{1}{2} \left( \frac{w''}{w'} \right)^2
\end{equation*}
is analytic, and the differential equation \eqref{eq:de2} with $A=S_w/2$ admits two linearly
independent solutions $f_1$ and $f_2$ such that $w=f_1/f_2$. Now, the complex 
 $a$-points of $w$ (i.e., solutions $z\in\D$ of $w(z)=a$) are either zeros of the solution $f_1 - a f_2$
or zeros of $f_2$, depending whether $a\in\C$ or $a=\infty$, respectively.


\begin{corollary} \label{cor:ulu}
If $w$ is meromorphic and locally univalent function in $\D$, and $|S_w(z)|^p (1-|z|^2)^{2p-1} \, dm(z)$ is a~Carleson measure
for some $0<p<\infty$, then the complex $a$-points of $w$ are uniformly separated for
any $a\in\C\cup \{\infty\}$.
\end{corollary}

We close our discussion with two examples.


\begin{example}
If $w$ is a locally univalent function in $\D$ such that $\log w'$ is in $\rm{BMOA}$, then it is easy to show that its Schwarzian derivative $S_w$
satisfies the assumption in the Corollary~\ref{cor:ulu}. We deduce that the preimage sequence $w^{-1} (a)$ of any point $a\in\C$ 
is uniformly separated. This fact has been proved in \cite[Lemma~10]{GNR:preprint},
and hence Corollary~\ref{cor:ulu} can be understood as a generalization of this result.
\end{example}


\begin{example}
If $w$ is a meromorphic function whose Schwarzian derivative $S_w$ is analytic and univalent in $\D$,
then $w$ satisfies the hypothesis of Corollary~\ref{cor:ulu}. The Carleson measure condition
follows from \cite[Theorem~11]{PR:2008}. Again, we conclude that all complex $a$-points of $w$
are uniformly separated for any $a\in\C\cup \{\infty\}$. 

This example implies that,
if $A$ is analytic and univalent in $\D$, then the zero-sequences of all non-trivial solutions 
of \eqref{eq:de2} are uniformly separated.
For more information on such differential equations we refer to \cite{CP:1972}.
\end{example}



\begin{thebibliography}{9}
%
\bibitem{CP:1972}
  J.A.~Cima and J.A.~Pfaltzgraff,
 \emph{Oscillatory behavior of $u′′(z)+h(z)u(z)=0$ for univalent $h(z)$}, 
 J. Analyse Math. \textbf{25} (1972), 311--322. 
%
\bibitem{G:2007}
  J.~Garnett,
  \emph{Bounded Analytic Functions},
  Revised first edition. Graduate Texts in Mathematics, 236. Springer, New York, 2007.
%
\bibitem{GNR:preprint}
  J.~Gr\"ohn, A.~Nicolau and J.~R\"atty\"a,
  \emph{Mean growth and geometric zero distribution of solutions of linear differential equations},
  submitted preprint, 2014.
%
\bibitem{N:1949}
     Z.~Nehari,
     \emph{The Schwarzian derivative and schlicht functions},
     Bull. Amer. Math. Soc. \textbf{55} (1949), 545--551.
%
\bibitem{PR:2008}
  F.~P\'erez-Gonz\'alez and J.~R\"atty\"a, 
  \emph{Univalent functions in Hardy, Bergman, Bloch and related spaces},
  J. Anal. Math. \textbf{105} (2008), 125--148. 
%
\bibitem{S:1955}
     B.~Schwarz,
     \emph{Complex nonoscillation theorems and criteria of univalence},
     Trans. Amer. Math. Soc. \textbf{80} (1955), 159--186.
%
\end{thebibliography}
\end{document}